\documentclass[12pt]{amsart}
\usepackage {amsfonts,amscd,amssymb,amsmath}
\usepackage{enumerate}

\newtheorem {prop}{Proposition}[section]
\newtheorem {lemme}[prop]{Lemma}
\newtheorem {theoreme}[prop]{Theorem}

\newtheorem {definition}[prop]{Definition}
\newtheorem{conjecture}[prop]{Conjecture}

\theoremstyle{definition}
\newtheorem{Examples}[prop]{Examples}

\newtheorem{Remark}[prop]{Remark}

\newcommand{\rk}{\operatorname{rk}}
\newcommand{\codim}{\operatorname{codim}}
\newcommand{\corank}{\operatorname{corank}}

\newcommand{\trd}{\mathrm{tr\, deg}}
\newcommand{\zset}{\mathbb Z}
\newcommand{\kset}{\Bbbk}
\newcommand{\ilie}{\mathfrak {i}}
\newcommand{\rlie}{\mathfrak {r}}
\newcommand{\blie}{\mathfrak {b}}
\newcommand{\glie}{\mathfrak {g}}
\newcommand{\hlie}{\mathfrak {h}}

\newcommand{\mlie}{\mathfrak {m}}
\newcommand{\ulie}{\mathfrak {u}}
\newcommand{\nlie}{\mathfrak {n}}
\newcommand{\qlie}{\mathfrak {q}}
\newcommand{\tlie}{\mathfrak {t}}
\newcommand{\alie}{\mathfrak {a}}

\newcommand{\calg}{{\mathcal G}}
\newcommand{\calh}{{\mathcal H}}
\newcommand{\cali}{{\mathcal I}}
\newcommand{\calz}{{\mathcal Z}}
\begin{document}

\title[On the index of the quotient]{On the index of the quotient of a Borel 
subalgebra by an ad-nilpotent ideal}

\author{C\'eline Righi}
\address{UMR 6086 du CNRS, D\'epartement de Math\'ematiques, Universit\'e de Poitiers, 
T\'el\'eport 2 - BP 30179, Boulevard Marie et Pierre Curie, 86962 Futuroscope Chasseneuil cedex, France.}
\email{Celine.Righi@math.univ-poitiers.fr}

\author{Rupert W.T. Yu}
\address{UMR 6086 du CNRS, D\'epartement de Math\'ematiques, Universit\'e de Poitiers, 
T\'el\'eport 2 - BP 30179, Boulevard Marie et Pierre Curie, 86962 Futuroscope Chasseneuil cedex, France.}
\email{yuyu@math.univ-poitiers.fr}

\begin{abstract}
In this paper, we give upper bounds for the index of the quotient
of the Borel subalgebra of a simple Lie algebra or its nilpotent radical 
by an ad-nilpotent ideal. 
For the nilpotent radical quotient, our bound is a generalization of the
formula for the index given by Panov in the type $A$ case. In general, this bound is not exact. 
Using results from Panov,  we show that the 
upper bound for the Borel quotient is exact in the type $A$ case, and 
we conjecture that it is exact in general.  
\vskip2em

\noindent{\sc R\'esum\'e.} 
Dans cet article, nous donnons des bornes sup\'erieures pour l'indice du
quotient d'une sous-alg\`ebre de Borel  d'une
alg\`ebre de Lie simple ou de son radical nilpotent par un id\'eal ad-nilpotent. 
Pour le quotient du radical nilpotent, notre borne sup\'erieure est une g\'en\'eralisation de la
formule obtenue pour l'indice par Panov  en type $A$. En g\'en\'eral,
cette borne n'est pas exacte. En utilisant des r\'esultats de Panov,
nous montrons qu'en ce qui concerne le quotient de la sous-alg\`ebre de
Borel, notre borne sup\'erieure est exacte dans le cas du type $A$, et
nous conjecturons que c'est aussi le cas en g\'en\'eral.
\end{abstract}
\maketitle

%---------------------------------------------------------------------------------------------------------
\section{Introduction}

The index of the quotient of a Borel subalgebra by an ad-nilpotent ideal is considered in recent 
works of P. Damianou, H. Sabourin and P. Vanhaecke \cite{DSV} on problems related to Toda-lattices 
and integrable systems. They associate to such an ad-nilpotent ideal an Hamiltonian system, and 
they want to determine whether this system is integrable. 

Let $\glie$ be a simple finite-dimensional Lie algebra over an algebraically closed field $\kset$
of characteristic zero, and $\blie$ a Borel subalgebra of $\glie$. 
For any ad-nilpotent ideal $\ilie$ of $\blie$, $(\blie/\ilie)^*$ is a Poisson submanifold of $\blie^*$. Its Poisson rank $L$ is equal to the dimension of $\blie/\ilie$ minus the index of $\blie/\ilie$. Since the number of equations required for the previous Hamiltonian system to be integrable is 
$\dim(\blie/\ilie)^{*}-L/2$, the calculation of the index of $\blie/\ilie$ is involved in this problem.

Recall that the index of a finite-dimensional Lie algebra $\alie$ over $\kset$
is the integer
$$
\chi (\alie)=\min_{f\in \alie^{*}} \dim \alie^{f}
$$
where for $f \in \mathfrak{g}^{*}$, we denote by
$\alie^{f}=\{ X\in \alie ; f([X,Y])= 0$ for all $Y \in \alie \}$,
the annihilator of $f$ for the coadjoint representation of $\alie$.
It is well-known that when $\alie$ is the Lie algebra of an algebraic group $A$,
$\chi (\alie )$ is the transcendence degree
of the field of $A$-invariant rational functions on $\alie^{*}$.

There are quite a lot of recent works on the computation of the index of certain classes of 
Lie subalgebras of a semisimple Lie algebra :
parabolic subalgebras and related subalgebras 
(\cite{DK}, \cite{Jos}, \cite{Panyseaweed}, \cite{TY2}, \cite{Mor2}), 
centralizers of elements and related subalgebras 
(\cite{Pany}, \cite{Char}, \cite{Yaki}, \cite{Mor1}).

Let $\hlie$ be a  Cartan subalgebra of the simple Lie algebra $\glie$ contained in $\blie$, 
$\Delta$ the associated root system, 
$\Delta^+$ the set of positive roots relative to $\blie$ and $\Pi=\{ \alpha_{1},\dots ,\alpha_{\ell} \}$ the 
corresponding set of simple roots. For each 
$\alpha \in \Delta$, let $ \mathfrak{g}^{ \alpha}$ be the root subspace of $\glie$ relative to $\alpha$.
Denote by $\nlie=\glie^{\Delta^+}$ the nilpotent radical of $\blie$
where for a subset $P$ of $\Delta^+$, we set
$$
\glie^{P} = \bigoplus_{\alpha \in P} \glie^{\alpha}.
$$

An ideal $\ilie $ of $\blie$ is ad-nilpotent if and only if for all $x\in \ilie$, $\mathrm{ad}_{\blie} x$ is 
nilpotent. Since any ideal of $\blie$ is $\hlie$-stable, we deduce easily that an ideal is ad-nilpotent 
if and only if it is nilpotent, and there exists a subset $\Phi\subset \Delta^+$ such that 
$\ilie =\glie^{\Phi}$. We set $\qlie_{\Phi}=\blie/ \ilie$ and $\mlie_{\Phi}=\nlie / \ilie$.

In \cite{Pa}, Panov determined the index of $\mlie_{\Phi}$, when
$\glie$ is simple of type $A$. His results are very explicit, and the index is completely
determined by $\Delta^+\setminus \Phi$ in a combinatorial way.  A similar
consideration of roots was used for the index of seaweed subalgebras in \cite{TY2}.
In this paper, we generalize these root combinatorial approaches to give upper bounds
for the index of $\qlie_{\Phi}$ and $\mlie_{\Phi}$ in all types.
Our upper bound for $\mlie_{\Phi}$ is not exact when $\glie$ is not of type $A$. 
However, using the results of Panov for $\mlie_{\Phi}$, we prove that our upper bound for 
$\qlie_{\Phi}$ is exact when $\glie$ is of type $A$, and we have not found so far any 
counter-examples in the other types. We give also a short discussion on the existence
of stable linear forms. 

\medskip
We shall recall a more general definition
of the index which is used in the paper. 
Let $\alie$ be the Lie algebra of an algebraic group $A$ and $V$
a rational $A$-module of finite dimension. The \textit{index} of $V$ is the integer 
$$
\begin{array}{rcl}
\chi(\alie, V) & = &
\dim V-\max\limits_{h\in V^*} \dim \alie.h =
\dim V-\max\limits_{h\in V^*} \codim_{\alie} \alie^h\\
& =& \trd_{\kset}(\kset(V^*)^A)
\end{array}
$$
where for $f\in V^*$, $\alie^f=\{X\in\alie;X.f=0\}$ and $\alie.f=\{X.f;X\in\alie\}$.
When $f\in V^*$ is such that 
$
\dim V-\dim \alie.f=\chi(\alie, V),
$
we say that $f$ is \textit{regular}. The set of regular elements of $V^*$ is a non-empty 
Zariski-open subset.

\section{H-sequences}\label{sec:H}

In this section, we introduce the combinatorial tools used to describe the upper bounds for the 
index of the quotients. This is a generalization of  the ``cascade'' construction of Kostant
(see for example \cite{TY2,TYL}) and the construction of Panov in type $A$ in \cite{Pa}. 

Recall the following standard partial order on $\Delta^+$. For $\alpha,\beta\in \Delta^+$, we have $\alpha \leqslant \beta$ if and only if $\beta-\alpha$ is a sum of positive roots. Let $E\subset \Delta^+$ and $\gamma\in E$. We set :
$$
H(E,\gamma)=\{\alpha\in E; \gamma-\alpha\in E\cup\{0\}\}.
$$
\begin{definition}
Let $E\subset \Delta^+$  and $\theta_1\in E$. We say that $(\theta_1)$ is an 
\emph{H-sequence of length $1$} in $E$ if $E=H(E, \theta_1)$.

By induction, for $\theta_1,\theta_2,\dots,\theta_r \in E$, we say that $(\theta_1,\theta_2,\dots,\theta_r)$ is an \emph{H-sequence of length} $r$ in $E$ if and only if :
\begin{enumerate}
\item[(i)] $(\theta_2,\theta_3,\dots,\theta_r)$ is an H-sequence of length $r-1$ in 
$E\setminus H(E, \theta_1)$,
\item[(ii)] $\theta_1$ is a maximal element for $\leqslant$ in $E$.
\end{enumerate}
\end{definition}

Let $(\theta_1,\theta_2,\dots,\theta_r)$ be an H-sequence of length $r$ in $E$. Set :
$$
E_1=E\setminus H(E, \theta_1)\quad,\quad \Gamma_1= H(E, \theta_1).
$$
For $i=1,\dots,r-1$, we set
$$
E_{i+1}=E_i\setminus H(E_i, \theta_{i+1})\quad,\quad \Gamma_{i+1}= H(E_i, \theta_{i+1}).
$$ 
It is clear from the definition that $E$ is the disjoint union of $\Gamma_1,\dots,\Gamma_r$, and we have $E_i=E_{i+1}\cup \Gamma_{i+1}$ for $i=0,\dots,r-1$, with the convention that $E_0=E$.

Let $\mathbf{h}$ be an H-sequence. We denote by $\ell(\mathbf{h})$ its length, $D(\mathbf{h})$ the vector space in $\hlie^*$ spanned by the elements in $\mathbf{h}$, and 
$d(\mathbf{h})=\dim D(\mathbf{h})$. 

\begin{Examples}\label{example1}

(i) Let $E=\Delta^+$. We recover (an ordered) Kostant's cascade construction of pairwise strongly orthogonal roots in $\Delta^+$. 

(ii) Let $\glie$ be of type $A_6$. Using the numbering of simple roots in \cite{TYL}, set $\alpha_{i,j}=\alpha_i+\dots+\alpha_j$. Take 
%$\cali =\{\alpha_{1,4},\alpha_{2,6}\}$, and 
$\Phi=\{\alpha\in\Delta^+; \alpha \geqslant \alpha_{1,4} \mbox{ or } \alpha \geqslant \alpha_{2,6}\}$, 
and $E=\Delta^+\setminus\Phi$. Then $\mathbf{h}=\{\alpha_{1,3},\alpha_{2,5},\alpha_{3,6},\alpha_{4,6},\alpha_{4,4}\}$ is an H-sequence of length $5$ in $E$, where 
$$
\begin{array}{ll}
\Gamma_1=\{\alpha_{1,3},\alpha_{1,1},\alpha_{2,3}, \alpha_{1,2},\alpha_{3,3}\},& \Gamma_4=\{\alpha_{4,6},\alpha_{4,5},\alpha_{6,6}\}, \\
\Gamma_2=\{\alpha_{2,5},\alpha_{2,2},\alpha_{3,5},\alpha_{2,4},\alpha_{5,5}\}, & \Gamma_5=\{\alpha_{4,4}\},\\
\Gamma_3=\{\alpha_{3,6},\alpha_{3,4},\alpha_{5,6}\}. \\
\end{array}
$$

We  have another H-sequence $\mathbf{h}'=(\alpha_{2,5},\alpha_{3,6},\alpha_{4,4},\alpha_{6,6},\alpha_{1,3},\alpha_{1,2},\alpha_{1,1}).$ which is of length $7$. Observe that $d(\mathbf{h})=5$ and $d(\mathbf{h}')=6$.
\end{Examples}

\noindent
\begin{lemme}\label{prop_sequence}
Let $E\subset\Delta^+$ and $\mathbf{h}=(\theta_1,\dots,\theta_r)$ be an H-sequence of length $r$ in $E$. Let $i,j, k \in \{1,\dots,r\}$.
\begin{enumerate}[(i)]
\item Let $\alpha\in \Gamma_i$ and $\beta\in \Gamma_j$ be such that $\alpha +\beta=\theta_k$. Then $k \geqslant \min(i,j)$.
\item There do not exist $i,j,k$ such that $\theta_i+\theta_j=\theta_k$.
\end{enumerate}
\end{lemme}

\begin{proof}
(i) If $k<\min(i,j)$, then $\alpha,\beta\in E_k$. It follows that $\alpha,\beta\in\Gamma_k$ and $k=j=i$, which contradicts the hypothesis. 

(ii) Assume that there  exist $i,j,k$ such that $\theta_i+\theta_j=\theta_k$. Then $\theta_k > \theta_i$ and $\theta_k > \theta_j$ and therefore by construction $k<\min(i,j)$, which contradicts the first point.
\end{proof}

%-----------------------------------------------------------------------------------------------------------------------
\section{Upper bounds for the index}

We give in this section upper bounds for the index of the quotients. The proof follows
closely to the one for the index of seaweed Lie algebras in \cite{TY2} even though 
we do not have the nice properties on the roots from the ``cascade construction''. 

Recall that if $\alie$ is a finite-dimensional Lie algebra over $\kset$ and
$f \in\alie^{*}$, we can define an alternating bilinear form $\Phi_{f}$ on $\alie$ by setting 
$$
\Phi_{f}(X,Y) = f([X,Y]),
$$
for $X,Y \in \alie$. Then $\alie^{f}=\{X\in\alie; \Phi_f(X,Y)=0, \mbox{ for all } Y\in\alie\}$ is the kernel of 
$\Phi_{f}$. Therefore we have
$$
\chi (\alie) = \min \{ \corank \Phi_f; f \in\alie^{*} \}.
$$

Let $\{ H_{1}, \dots , H_{\ell}\}$ be a basis of $\mathfrak{h}$. For $\alpha\in\Delta$, we denote by $X_{\alpha}$ a non-zero element of $\mathfrak{g}^{ \alpha}$. Then $\{ H_{i}; 1 \leqslant i \leqslant
\ell\} \cup \{ X_{\alpha}; \alpha \in \Delta\}$ is a basis of $\mathfrak{g}$ and we shall denote by $\{ H_{i}^{*}; 1 \leqslant i
\leqslant \ell\} \cup \{ X_{\alpha}^{*}; \alpha \in \Delta\}$ the corresponding dual basis.

Let $\Phi$ be a subset of $\Delta^+$ such that $\ilie=\mathfrak{g}^{ \Phi}$ is an ad-nilpotent ideal of $\blie$. Suppose that $\mathbf{h}=(\theta_1, \dots, \theta_s)$ is an H-sequence of length $s$ of $E=\Delta^+\setminus \Phi$. We have the following $\hlie$-module isomorphisms :
$$
\qlie_{\Phi}\simeq\hlie \oplus \glie^{\Delta^+\setminus \Phi}\quad,\quad
\mlie_{\Phi}\simeq\glie^{\Delta^+\setminus \Phi}.
$$

Let $\mathbf{a}=(a_1,\dots,a_s)$ be an element of $(\kset^*)^s$. Identifying $\qlie_{\Phi}^*$ with $\hlie^*\oplus \displaystyle\sum_{\alpha\in\Delta^{+}\setminus \Phi} \kset X_ {\alpha}^*$, 
we define the following element of $\qlie_{\Phi}^*$ :
$$
f_{\mathbf{a}}=\sum_{i=1}^s a_i X_{\theta_i}^*.% \in \qlie_{\Phi}^*.
$$

We fix a total order $<$ on $\Delta^+$ compatible with the partial order $\leqslant$. 

For $i \in\{1,\dots,s\}$, set 
$$
\calg_i=\{(\alpha, \beta)\in \Gamma_i\times\Gamma_i;\alpha+\beta=\theta_i \mbox{ and } \alpha<\beta\}\quad,\quad t_i=\sharp\calg_i \quad, 
$$
and
$$
\calg =\bigcup_{i=1}^s \calg_i \quad,\quad t= \sharp\calg.
$$
Denote by $\calz$  the set of pairs $(\alpha, \beta)$ of $E^2$ such that $\alpha<\beta$ and there exists $k\in\{1,\dots,s\}$ satisfying $\alpha+\beta=\theta_k$.

\medskip
For $z=(\alpha,\beta)\in\calz$, we set 
$$
v_z=X^*_{\alpha}\wedge X^*_{\beta}\in \bigwedge\nolimits^2 \qlie_{\Phi}^* .
$$
Identifying $\Phi_{f_{\mathbf{a}}}$ with an element of $\bigwedge^2 \qlie_{\Phi}^*$, we have
$$
\Phi_{f_{\mathbf{a}}}=\Psi_{f_{\mathbf{a}}}+ \Theta_{f_{\mathbf{a}}}
$$
where 
$$
\Theta_{f_{\mathbf{a}}} =\sum_{i=1}^s K_i\wedge X_{\theta_i}^* \quad,\quad \Psi_{f_{\mathbf{a}}}=\sum_{z\in\calz} \lambda_z v_z \quad,
$$
with $K_i\in D(\mathbf{h})$ for $i=1,\dots,s$ and $\lambda_z\in \kset$ for all $z\in\calz$.

If $z = (\alpha , \beta ) \in \mathcal{Z}$, then $\Theta_{f_{\mathbf{a}}}(X_{\alpha}, X_{\beta}) = 0$. Moreover, we have $[X_{\alpha}, X_{\beta}] = \mu_{z}X_{\theta_i}$, for some $i\in\{1,\dots,s\}$ and a non-zero scalar $\mu_{z}$. Consequently,
\begin{equation}\label{eq1}
\lambda_{z} =\Phi_{f_{\mathbf{a}}} (X_{\alpha},X_{\beta}) = f_{\mathbf{a}}([X_{\alpha}, X_{\beta}]) = \mu_{z} a_{i}. 
\end{equation}
Thus $\lambda_{z}$ is non-zero.

\begin{lemme}\label{lemma1} In the above notations :
\begin{enumerate}[(i)]
\item $\qlie_{\Phi}^{f_{\mathbf{a}}}$ contains a commutative subalgebra of $\qlie_{\Phi}$, 
consisting of semi-simple elements and of dimension $
\ell-d(\mathbf{h}).$
\item We have $\bigwedge^{d(\mathbf{h})} \Theta_{f_{\mathbf{a}}} \ne 0$ and $\bigwedge^{d(\mathbf{h})+1}
\Theta_{f_{\mathbf{a}}} = 0$. 

\item There exists a non-empty open subset $U$ of $(\kset^*)^{s}$ such that we have 
$\bigwedge^{t} \Psi_{f_{\mathbf{a}}} \ne 0$ and $\bigwedge^{d(\mathbf{h})+t} \Phi_{f_{\mathbf{a}}} \ne
0$ whenever $\mathbf{a} \in U$.
\end{enumerate}
\end{lemme}

\begin{proof} The proof is similar to the one for the Lemme in \cite[\S 3.9]{TY2}.

(i) For simplicity, we write $\qlie=\qlie_{\Phi}$. Let $\tlie=\{x\in\hlie;\theta_i(x)=0 \mbox{ for } i=1,\dots, s\}$ be the orthogonal of $D(\mathbf{h})$ in $\hlie$. Then :
$$
\dim \tlie=\dim \hlie - \dim D(\mathbf{h}) = \ell - d(\mathbf{h}).
$$
We also have that $\displaystyle [\tlie,\qlie]\subset \bigoplus_{\alpha\in E\setminus \{\theta_1,\dots,\theta_s\}}\glie^{\alpha}$. It follows that $\tlie$ is contained in $\qlie^{f_{\mathbf{a}}}$, and therefore, we obtain the result.

\medskip
(ii) Set $r=d(\mathbf{h})$. Let $\cali=\{ i_{1},\dots ,i_{r}\} \subset\{1,\dots,s\}$ be such that 
$(\theta_{i_{1}}, \dots ,\theta_{i_{r}})$ is a basis of $D(\mathbf{h})$ and complete to a basis 
$\mathcal{B}' = (\beta_{1}, \dots ,\beta_{\ell})$ of $\mathfrak{h}^{*}$ such that $\beta_k=\theta_{i_{k}}$ 
for $k=1,\dots ,r$.
Denote by $\mathcal{B} = (h_{1},\dots , h_{\ell})$ the basis of $\mathfrak{h}$ dual to $\mathcal{B}'$. 
Then we have,
$$
f_{\mathbf{a}}([h_k ,X_{\theta_j}])=\left\{
\begin{array}{ll}
a_{i_{k}} &\mbox{if } j=i_{k}, \\
0 &\mbox{otherwise}.\\
\end{array}
\right.
$$
It follows that 
$$
\Theta_{f_{\mathbf{a}}}=\sum_{i\in\cali} a_{i} \theta_{i}\wedge X_{\theta_{i}}^*
+\sum_{j\not\in\cali} K_j\wedge X_{\theta_{j}}^*,
$$
where $K_j\in D(\mathbf{h})$. The result follows easily beacause $\sharp\cali=d(\mathbf{h})$.

\medskip
(iii) If $z,z' \in \mathcal{Z}$, we have $v_{z} \wedge v_{z'} = v_{z'} \wedge v_{z}$ and $v_{z}
\wedge v_{z} = 0$. Let $z_{1}, \dots , z_{n}$ be the elements of $\mathcal{Z}$ such that $z_{1}, \dots , z_{t}$ are the elements of $\calg$. For simplicity, let us write $\lambda_{i} v_{i}$ for $\lambda_{z_{i}} v_{z_{i}}$ and $\mu_{i}$ for $\mu_{z_{i}}$. Consequently:
$$
\bigwedge\nolimits^{t} \Psi_{f_{\mathbf{a}}} = t! \operatornamewithlimits{\textstyle \sum}_{1
\leqslant i_{1} < \cdots < i_{t} \leqslant n} \lambda_{i_{1}} \cdots
\lambda_{i_{t}} v_{i_{1}} \wedge \cdots \wedge v_{i_{t}}. 
$$

In the previous sum, the coefficient of $v_{1} \wedge
\cdots \wedge v_{t}$ is by \eqref{eq1}
$$
\operatornamewithlimits{\textstyle \prod}_{i=1}^s
a_{i}^{t_{i}} \Big(\operatornamewithlimits{\textstyle \prod}_{z \in\calg} \mu_{z}\Big).
$$

Now assume that $v_{i_{1}} \wedge \cdots \wedge v_{i_{t}} = \lambda
v_{1} \wedge \cdots \wedge v_{t}$, with $\lambda \in \kset^*$, where $i_{1} < \cdots < i_{t}$ and $(i_{1}, \dots , i_{t})
\ne (1, \dots , t)$. 

If $z=(\alpha,\beta)\in\calz$, we denote by $\widetilde{z}=\{\alpha,\beta\}$ the underlying set of $z$. Then the set $\mathcal{S} = \widetilde{z}_{1} \cup
\cdots \cup \widetilde{z_{t}}$ is the disjoint union of the sets $\widetilde{z_{i_{k}}}$ for $1 \leqslant k \leqslant t$. It follows that if $z_{i_{k}} \notin \calg$, then we have  by Lemma \ref{prop_sequence} that there exist $i,j\in\{1,\dots,s\}$ such that $i\not=j$ and $z_{i_{k}}=(\alpha,\beta)$ where $(\alpha,\beta)\in(\Gamma_i\setminus\{\theta_i\})\times(\Gamma_j\setminus\{\theta_j\})$. 

Let $\cali=\{ k; z_{i_k}\not\in\calg\}$. Let $i_0$ be minimal among the elements $j\in\{1,\dots,s\}$ 
verifying :
$$
\left(\Gamma_{j}\setminus\{\theta_j\}\right) \cap \left(\bigcup_{k\in\cali} \widetilde{z_{i_k}}\right) \ne \emptyset .
$$ 
Then there exists $\alpha\in\Gamma_{i_0}$, $k\in\cali$ and $\beta\in\Delta^+$ such that $\widetilde{z_{i_k}}=\{\alpha,\beta\}$. By our choice of $i_0$ and since $z_{i_k}\not\in\calg$, there exists $j\geqslant i_0$ and $l\in\{1,\dots,s\}$ such that $\beta\in\Gamma_j$ and $\alpha+\beta=\theta_l$. Then, by Lemma \ref{prop_sequence}, we have $l>\min(i_0,j)=i_0$.  It follows that $\lambda_{z_{i_k}} = \mu_{z_{i_k}}a_{l}$, where $l\ne i_0$.

We deduce that the coefficient of  $v_{i_{1}} \wedge \cdots
\wedge v_{i_{t}}$ in the sum giving $\bigwedge^{t} \Psi_{f_{\mathbf{a}}}$ is of the form 
$$
\mu_{i_{1}} \cdots \mu_{i_{t}} \operatornamewithlimits{\textstyle
\prod}_{i=1}^s a_{i}^{m_{i}},
$$
with $m_{i_{0}} < t_{i_{0}}$. 

It is now clear that there exists a non-empty open subset $U$ of $(\kset^*)^{s}$ verifying $\bigwedge^{t} \Psi_{f_{\mathbf{a}}} \ne 0$
if $\mathbf{a} \in U$. 

Finally, we have 
$$
\bigwedge\nolimits^{r+t} \Phi_{f_{\mathbf{a}}} = \operatornamewithlimits{\textstyle
\sum}_{k=0}^{r+t} \begin{pmatrix}
r+t \\
k\end{pmatrix} \left(\bigwedge\nolimits^{k} \Psi_{f_{\mathbf{a}}} \right) \wedge \left(\bigwedge\nolimits^{r+t-k}
\Theta_{f_{\mathbf{a}}}\right). 
$$

Set $\displaystyle\ulie=\sum_{i=1}^s \kset X_{\theta_i}$. Since $\bigwedge^{j}\Theta_{f_{\mathbf{a}}} \in (\textstyle{\bigwedge^{j}}
D(\mathbf{h}))\wedge (\textstyle{\bigwedge^{j}}\mathfrak{u}^{*})$, to show that $\bigwedge^{r+t} \Phi_{f_{\mathbf{a}}} \ne 0$, it suffices to prove that $(\bigwedge^{t} \Psi_{f_{\mathbf{a}}})\wedge (\bigwedge^{r} \Theta_{f_{\mathbf{a}}}) \ne 0$.

If $\mathbf{a}\in U$, then we deduce from the preceding paragraphes  that 
$$
\bigwedge\nolimits^{t} \Psi_{f_{\mathbf{a}}} = \lambda v_{1}\wedge \cdots \wedge v_{t} + w,
$$
where $\lambda \in \kset^{*}$ and $w$ is a linear combination of elements 
of the form $v_{z_{i_{1}}} \wedge \cdots \wedge
v_{z_{i_{t}}}$, with $\widetilde{z_{i_{1}}} \cup \cdots \cup
\widetilde{z_{i_{t}}} \ne \mathcal{S}$. It is therefore clear that $(\bigwedge^{t} \Psi_{f_{\mathbf{a}}}) \wedge (\bigwedge^{r} \Theta_{f_{\mathbf{a}}}) \ne 0$ if $\mathbf{a} \in U$.
\end{proof}

\begin{theoreme}\label{th_index}
Let $\Phi$ be a subset of $\Delta^+$ such that $\glie^{\Phi}$ is an ad-nilpotent ideal of $\glie$. Denote by $\calh$ the set of H-sequences of $\Delta^+\setminus\Phi$. Then, we have
$$
\begin{array}{l}
\chi (\qlie_{\Phi}) \leqslant \mathrm{min}\{\ell + \ell(\mathbf{h})  - 2 d(\mathbf{h}) ;\mathbf{h}\in\calh\},\\
\chi (\mlie_{\Phi}) \leqslant  \mathrm{min}\{\ell(\mathbf{h});\mathbf{h}\in\calh\}.
\end{array}
$$
\end{theoreme}

\begin{proof} 
Let $\mathbf{h}\in\calh$ with $\ell(\mathbf{h})=s$ and $t$ as defined in the beginning of this section. By definition, we have 
$$
\begin{array}{l}
\dim \qlie_{\Phi} = \dim \mathfrak{h} + \ell(\mathbf{h}) + 2t,\\
\dim \mlie_{\Phi} = \ell(\mathbf{h}) + 2t.
\end{array}
$$
Let $U$ be a non-empty open subset of $(\kset^*)^s$ verifying part (iii) of Lemma \ref{lemma1}. If $\mathbf{a} \in U$, then the fact that $\bigwedge^{d(\mathbf{h})+t} \Phi_{f_{\mathbf{a}}} \ne 0$ implies that $\rk (\Phi_{f_{\mathbf{a}}}) \geqslant 2(d(\mathbf{h})+t)$. Thus 
$$
\dim \qlie_{\Phi}^{f_{\mathbf{a}}} \leqslant \dim \qlie_{\Phi} - 2 (d(\mathbf{h})+t).
$$
Hence 
$$
\dim \qlie_{\Phi}^{f_{\mathbf{a}}} \leqslant \dim \hlie + \ell(\mathbf{h}) - 2d(\mathbf{h}).
$$
In the same manner, if $\mathbf{a} \in U$, then the fact that $\bigwedge^{t} \Psi_{f_{\mathbf{a}}} \ne 0$ implies that $\rk (\Psi_{f_{\mathbf{a}}}) \geqslant 2t$. Thus 
$$
\dim \mlie_{\Phi}^{f_{\mathbf{a}}} \leqslant \dim \mlie_{\Phi} - 2t = \ell(\mathbf{h}).
$$
So we are done.
\end{proof}

For an H-sequence $\mathbf{h}$, we define
$$
c(\mathbf{h}) = \ell + \ell(\mathbf{h})  - 2 d(\mathbf{h}). 
$$

\begin{prop}\label{prop01}
Let us conserve the notations of Theorem \ref{th_index}. If $\mathbf{h}\in\calh$ verifies $c(\mathbf{h}) \in \{ 0,1\}$, then $\chi (\qlie_{\Phi}) = c(\mathbf{h})$.
\end{prop}

\begin{proof} The case $c(\mathbf{h}) = 0$ is clear by Theorem \ref{th_index}. So let us suppose that 
$c(\mathbf{h}) = 1$.

We have $\dim \qlie_{\Phi} - c(\mathbf{h}) = 2(d(\mathbf{h})+t)$. Since $\dim
\qlie_{\Phi} - \chi (\qlie_{\Phi})$ is an even integer (it is the rank of an alternating bilinear form on 
$\qlie_{\Phi}$), we deduce that $c(\mathbf{h})$ and $\chi
(\qlie_{\Phi})$ are of the same parity. So $\chi
(\qlie_{\Phi}) = 1$. 
\end{proof}

%%%%%%%%%%%%%%%%%%%%%%%%%%%%%%%%%%%%%%%%%%%%%%%%%%%%%%%%%%%%%%%%%%%%%%%%%%%%%%%%%%%%%%%%%%%%%%%%%%%%%%%%%%%%%%%%%%%%%%%%%%%%%%%%%%%%%%%%%%%%%%%%%%%%%
\section{Type A}

Let us assume in this section 
that $\glie$ is of type $A_{\ell}$.
We shall show that the upper bound for the index of $\qlie_{\Phi}$ is exact in this special case.

We fix a subset $\Phi$ of $\Delta^+$ such that $\ilie=\glie^{\Phi}$ is an ad-nilpotent ideal of $\glie$. As in example \ref{example1}, we use the numbering of simple roots in \cite{TYL}, 
and we set $\alpha_{i,j}=\alpha_i+\dots+\alpha_j$ when $i\leqslant j$. 

We fix the following total order $\prec$ on $\Delta^{+}$ compatible with the partial order $\leqslant$ :
$$
\alpha_{1,\ell}\succ\alpha_{1,\ell-1}\succ\dots\succ\alpha_{1,2}\succ\alpha_{1,1}\succ\alpha_{2,\ell}\succ\alpha_{2,\ell-1}\succ\dots\succ\alpha_{\ell-1,\ell}\succ\alpha_{\ell,\ell}.
$$
It is clear that there is a unique H-sequence $\mathbf{h}=(\theta_1,\dots,\theta_s)$ of $E=\Delta^+\setminus\Phi$, satisfying $\theta_1\succ\theta_2\succ\dots\succ\theta_s$. This H-sequence is 
considered by Panov in \cite{Pa}, and we shall call this H-sequence 
the \textit{Panov H-sequence} of $E$.

Using the notation of section \ref{sec:H}, for $j=1,\dots,s$, set :
$$
\nlie_j=\bigoplus_{\alpha\in E_j\cup \Phi}\glie^{\alpha}\quad,\quad \mlie_j=\nlie_j/\ilie.
$$

In \cite{Pa}, Panov proved that for $j=1,\dots,s$, $\nlie_j$ and $\mlie_j$ are Lie subalgebras. Consider the localization $S(\mlie_{j-1})_{X_{\theta_{j}}}$ of the algebra $S(\mlie_{j-1})$ with respect to the multiplicative subset generated by $X_{\theta_{j}}$. He defined an embedding of Poisson algebras $\Psi_{j-1} : S(\mlie_j)\rightarrow S(\mlie_{j-1})_{X_{\theta_{j}}}$. Moreover, one observes directly from the definition of $\Psi_{j-1}$ that it is $\hlie$-equivariant. 

Extending the $\Psi_j$ with the appropriate localizations, we set 
$$
f_1=X_{\theta_1} \quad , \quad
f_j=\Psi_{0}\circ\dots\circ\Psi_{j-2}(X_{\theta_j}) \mbox{ for } 2\leqslant j\leqslant s.
$$

Panov proved that the $f_j$'s are algebraically independent elements of $\kset(\mlie_{\Phi}^*)^{\mlie_{\Phi}}$. More precisely, we have $\kset(\mlie_{\Phi}^*)^{\mlie_{\Phi}}=\kset(f_1,\dots,f_s)$,
and hence $\chi (\mlie_{\Phi})=s$. 
Furthermore, using the fact that the embeddings $\Psi_{j-1}$ are $\hlie$-equivariant, 
the element $f_j$ is of weight $\theta_j$ for $j=1,\dots ,s$.

%Set $r=d(\mathbf{h})$. 
Let $\cali\subset\{1,\dots,s\}$ be such that $\{\theta_i;i\in\cali\}$ is a basis of $D(\mathbf{h})$. For 
$j\in\{1,\dots ,s\}\setminus\cali$, we have 
$$
\lambda_j\theta_j=\sum_{i\in\cali} \lambda_i\theta_i
$$
where $\lambda_i\in\zset$ for $i\in\cali$ and $\lambda_j\in\zset^*$.

Set
$$
g_j=\left(\prod_{i\in\cali} f_i^{\lambda_i}\right)f_j^{-\lambda_j}\in\kset(\mlie_{\Phi}^*)^{\mlie_{\Phi}}.
$$
By construction, the elements $g_j$ are of weight zero. Hence $g_j\in\kset(\mlie_{\Phi}^*)^{\qlie_{\Phi}}$. Since the elements $f_1,\dots,f_s$ are algebraically independent, it follows that 
\begin{equation}\label{minoration}
\chi(\qlie_{\Phi},\mlie_{\Phi})=\mathrm{tr\, deg}_{\kset}(\kset(\mlie_{\Phi}^*)^{\qlie_{\Phi}})\geqslant \ell(\mathbf{h})-d(\mathbf{h}).
\end{equation}

\begin{theoreme}\label{th_indice_rep}
Let $\mathbf{h}=(\theta_1,\dots,\theta_s)$ be the Panov H-sequence of $E$. There exists a non-empty open subset $U$ of $(\kset^*)^s$ such that $\dim\qlie_{\Phi}^{f_{\mathbf{a}}}=c(\mathbf{h})$ whenever $\mathbf{a} \in U$. Moreover, we have 
$$
\chi(\qlie_{\Phi},\mlie_{\Phi})=\ell(\mathbf{h})-d(\mathbf{h}).
$$
\end{theoreme}

\begin{proof}
Let $U$ be a non-empty open subset of $(\kset^*)^s$ verifying part (iii) of Lemma \ref{lemma1}. 
Set $r=d(\mathbf{h})$. Let $\cali=\{ i_{1},\dots ,i_{r}\} \subset\{1,\dots,s\}$ be such that 
$(\theta_{i_{1}}, \dots ,\theta_{i_{r}})$ is a basis of 
$D(\mathbf{h})$ and complete to a basis $\mathcal{B}' = (\beta_{1}, \dots ,\beta_{\ell})$ of 
$\mathfrak{h}^{*}$ such that $\beta_k=\theta_{i_{k}}$ for $k=1,\dots ,r$. Denote by 
$\mathcal{B} = (h_{1},\dots , h_{\ell})$ the basis of $\mathfrak{h}$ dual to $\mathcal{B}'$.

Let $m=\dim\mlie_{\Phi}$ and $\mathcal{C}$ be a basis of $\mlie_{\Phi}$. Then the matrix of $\Phi_{f_{\mathbf{a}}}$ in the basis $\mathcal{B}'\cup\mathcal{C}$ is 
$$
M = \begin{pmatrix}
0_{\ell,\ell} & A\\
-{}^{t}A & B
\end{pmatrix},
$$
where $A$ is an element of rank $d(\mathbf{h})$ in the set of $\ell\times m$ matrices 
$\mathcal{M}_{\ell,m}(\kset )$, and 
$B \in \mathcal{M}_{m,m}(\kset )$ the set of $m\times m$ matrices. 
Set 
$$
M' = \begin{pmatrix}
A\\
B
\end{pmatrix}.
$$
Then by \eqref{minoration}, we have
\begin{equation}\label{eqrg1}
\dim\mlie_{\Phi}-\rk(M')\geqslant \chi(\qlie_{\Phi},\mlie_{\Phi})\geqslant \ell(\mathbf{h})-d(\mathbf{h}).
\end{equation}
It follows that 
\begin{equation}\label{eqrg2}
\rk(M')\leqslant \dim\mlie_{\Phi}-\ell(\mathbf{h})+d(\mathbf{h}) \quad,
\end{equation}
and since $\rk(M)\leqslant \rk(A)+\rk(M')$, we deduce that 
\begin{equation}\label{eqrg3}
\rk(M)\leqslant \dim\mlie_{\Phi}-\ell(\mathbf{h})+2d(\mathbf{h}).
\end{equation}
Hence, 
$$
\dim \qlie_{\Phi}^{f_{\mathbf{a}}}=\dim\qlie_{\Phi}-\rk(M)\geqslant c(\mathbf{h}).
$$
It follows by Theorem \ref{th_index} that $\dim\qlie_{\Phi}^{f_{\mathbf{a}}}=c(\mathbf{h})$ and we have equalities in \eqref{eqrg1} and \eqref{eqrg2}, \eqref{eqrg3}. Consequently $\rk(M')=\dim\mlie_{\Phi}-\ell(\mathbf{h})+d(\mathbf{h})$ and $\chi(\qlie_{\Phi},\mlie_{\Phi})=\ell(\mathbf{h})-d(\mathbf{h}).$
\end{proof}

\begin{theoreme}\label{th_indice}
Let $\mathbf{h}=(\theta_1,\dots,\theta_s)$ be the Panov H-sequence of $E$. Then we have $\chi(\qlie_{\Phi})=\ell-\ell(\mathbf{h})+2d(\mathbf{h})$.
\end{theoreme}

\begin{proof}
Let $U$ be a non-empty open subset of $(\kset^*)^s$ verifying part (iii) of Lemma \ref{lemma1}. 
Let $S$ be the subset of $\mlie_{\Phi}^*$ consisting of elements of the form
$$
f_{\boldsymbol{\lambda}}=\sum_{i=1}^s \lambda_i X_{\theta_i}^*
$$
for $\boldsymbol{\lambda}=(\lambda_1,\dots,\lambda_s)\in\kset^s$. Then $\Omega=\{f_{\mathbf{a}} ;\mathbf{a}\in U\}$ is an open subset of $S$.

Let $M$ be the algebraic adjoint group of $\mlie_{\Phi}$. Consider the elements $z_1,\dots,z_s$ of $\kset[\mlie_{\Phi}^*]^M$ constructed by Panov in \cite{Pa} such that $\kset(\mlie_{\Phi}^*)^M=\kset(z_1,\dots,z_s)$, and
\begin{equation}\label{eq2}
z_i=X_{\theta_i}P+R,
\end{equation}
where $P$ is some product of powers of $z_1,\dots,z_{i-1}$ and $R$ is a polynomial in $X_{\alpha}$ for $\alpha\succ\theta_i$.

For $i=1,\dots,s$, denote by $U_i=\{f\in\qlie_{\Phi}^*;z_i(f)\not=0\}$ the standard open subset of 
$\qlie_{\Phi}^*$ associated to $z_i$. By \eqref{eq2}, $z_1=X_{\theta_1}$ so we clearly have 
$U_1\cap \Omega\not= \emptyset$. Next, for $i>0$, we have
$$
z_{i+1}(f_{\boldsymbol{\lambda}})=
\lambda_{i+1}P(f_{\boldsymbol{\lambda}})+R(f_{\boldsymbol{\lambda}}).
$$
By the properties of $P$ and $R$ from the preceding paragraph, $P$ depends only
on $z_{1},\dots ,z_{i}$, and $R(f_{\boldsymbol{\lambda}})$ depends only on 
$\lambda_{1},\dots ,\lambda_{i}$. By induction, we obtain
$$
\Omega\cap\left(\bigcap_{j=1}^{i+1} U_{j}\right)\neq \emptyset.
$$
Hence $\displaystyle\Omega'=\Omega\cap\left(\bigcap_{i=1}^s U_i\right)$ is a 
non-empty open subset of $\Omega$.

Consider the map
$$
\begin{array}{ccclcr}
\Psi &: &M\times\Omega' &\rightarrow& \mlie^* \\
&&(m,f)&\mapsto &m.f \\
\end{array}
$$

Assume that $f_{\boldsymbol{\lambda}}$ and $f_{\boldsymbol{\mu}}$ are two elements of $\Omega'$ which are in the same $M$-orbit. Then we have 
$z_i(f_{\boldsymbol{\lambda}})=z_i(f_{\boldsymbol{\mu}})$
for $i=1,\dots,s$. In particular, we have $z_1(f_{\boldsymbol{\lambda}})=
z_1(f_{\boldsymbol{\mu}})$ so $\lambda_1=\mu_1$. 

Let us proceed by induction. Suppose that $i > 0$ and that $\lambda_{j}=\mu_{j}$
for $1\leqslant j\leqslant i$. We have
\begin{eqnarray*}
z_{i+1}(f_{\boldsymbol{\lambda}})&=&\lambda_{i+1}P(f_{\boldsymbol{\lambda}})+
R(f_{\boldsymbol{\lambda}}),\\
z_{i+1}(f_{\boldsymbol{\mu}})&=&\mu_{i+1}P(f_{\boldsymbol{\mu}})+R(f_{\boldsymbol{\mu}}).
\end{eqnarray*}
Since $f_{\boldsymbol{\lambda}},f_{\boldsymbol{\mu}}\in\Omega'$, 
we deduce from the properties of $P$ and $R$ described above that $\lambda_{i+1}=\mu_{i+1}$.
Hence $\boldsymbol{\lambda}=\boldsymbol{\mu}$. 

It follows that for any $f_{\boldsymbol{\lambda}}\in\Omega'$, we have 
$$
\Psi^{-1}(f_{\boldsymbol{\lambda}})=\{(m,g)\in M\times\Omega'; m.g=f_{\boldsymbol{\lambda}}\}\simeq \mathrm{Stab}_M (f_{\boldsymbol{\lambda}}).
$$
By \cite{Pa} and Theorem \ref{th_index}, we have 
$\dim \mlie_{\Phi}^{f_{\boldsymbol{\lambda}}}=\ell(\mathbf{h})=s$. Since 
$\dim \mlie_{\Phi}^{f_{\boldsymbol{\lambda}}}=\dim \mathrm{Stab}_M (f_{\boldsymbol{\lambda}})$ and 
$\dim \mathrm{Im}\Psi=\dim(M\times\Omega')-\dim \Psi^{-1}(f_{\boldsymbol{\lambda}})$, 
we obtain that $\dim \mathrm{Im}\Psi=\dim M$. Therefore $M.\Omega'$ contains an open 
subset $\mathcal O$ of $\mlie_{\Phi}^*$.

Let $p$ be the projection $\qlie_{\Phi}^*\rightarrow \mlie_{\Phi}^*$ via restriction. Then $p$ is $M$-equivariant. Since the set of regular elements of $\qlie_{\Phi}^*$ is an open subset, we deduce that 
there exist $\varphi\in p^{-1}(\mathcal O)$ and $f_{\boldsymbol{\lambda}}\in\Omega'$ such that 
$\varphi$ is regular and 
$$
\varphi_{|\mlie_{\Phi}}=f_{\boldsymbol{\lambda}}{}_{|\mlie_{\Phi}}.
$$
For all $X,Y\in\qlie_{\Phi}$, we have  $\varphi([X,Y])=f_{\boldsymbol{\lambda}}([X,Y])$, and therefore
$$
\mathrm{Mat} (\Phi_{\varphi})=\begin{pmatrix}
0_{\ell,\ell} & A\\
-{}^{t}A & B
\end{pmatrix}=
\mathrm{Mat} (\Phi_{f_{\boldsymbol{\lambda}}}).
$$
Hence by Theorem \ref{th_indice_rep}, we have 
$$
\chi(\qlie_{\Phi})=\dim \qlie_{\Phi}^{f_{\boldsymbol{\lambda}}}=\ell+\ell(\mathbf{h}) -2d(\mathbf{h}).
$$
\end{proof}

\begin{Remark}
According to the previous theorem, 
$c(\mathbf{h})$ is minimal when $\mathbf{h}$ is the Panov H-sequence.

Let $\glie$ be of type $A_6$. 
Set $\Phi=\{\alpha\in\Delta^+; \alpha \geqslant \alpha_{2,5}\}$, and 
$E=\Delta^+\setminus\Phi$. Then the Panov H-sequence of $E$ is 
$\mathbf{h}=\{\alpha_{1,4}, \alpha_{2,3}, \alpha_{3,5}, \alpha_{4,6}, \alpha_{5,6}, \alpha_{5,5}\}$ and 
we have $\ell(\mathbf{h})=d(\mathbf{h}) =6$.

We  have another H-sequence 
$\mathbf{h}'=(\alpha_{3,5},\alpha_{4,6},\alpha_{6,6},\alpha_{1,4},\alpha_{1,3},\alpha_{1,2},
\alpha_{2,3},$ $\alpha_{2,2})$ such that $\ell(\mathbf{h}')=8$ and $d(\mathbf{h}')=6$. 
Observe that $c(\mathbf{h})=0$ and $c(\mathbf{h}')=2$.
\end{Remark}

%---------------------------------------------------------------------------------------------------------------------------------------

\section{Stability}

Let $\mathfrak{a}$ be an  algebraic Lie algebra and let $A$ be its adjoint algebraic group. 
Recall that $g \in\mathfrak{a}^{*}$ is \emph{stable} if there exists an open subset $U$ of 
$\mathfrak{a}^{*}$ containing $g$ such that $\mathfrak{a}^{g}$ and $\mathfrak{a}^{h}$ are $A$-conjugate for all $h \in U$. 

The following result is proved in \cite{TY} :
\begin{prop}\label{stable}
Let $\mathfrak{a}$ be an  algebraic Lie algebra and $f \in\mathfrak{a}^{*}$. 
\begin{enumerate}[(i)]
\item If $f$ is stable, then it is a regular element of $\mathfrak{a}^{*}$.
\item The linear form $f$ is stable if and only if $[\alie,\alie^f]\cap\alie^f=\{0\}$.
\end{enumerate}
\end{prop}

In this section, we return to the general case, that is $\glie$ is not necessarily of type $A$.
\begin{prop}\label{stable:prop2}
Let $\Phi$ be a subset of $\Delta^+$ such that $\glie^{\Phi}$ is an ad-nilpotent ideal of $\glie$. Let $\mathbf{h}=(\theta_1,\dots,\theta_s)$ be an H-sequence of $\Delta^+\setminus\Phi$ consisting of linear independent elements. Then there exists $f\in\qlie_{\Phi}^*$ which is stable and 
$\chi (\qlie_{\Phi}) = c(\mathbf{h})$.
\end{prop}

\begin{proof}
Our hypothesis implies that $\ell(\mathbf{h})=d(\mathbf{h})$. So
$$
c(\mathbf{h})=\ell-d(\mathbf{h}).
$$
Let $U$ be a non-empty open subset of $(\kset^*)^s$ verifying part (iii) of Lemma \ref{lemma1}. 
By Theorem \ref{th_index} and Lemma \ref{lemma1}, if $\mathbf{a}\in U$, we have 
$\dim \qlie_{\Phi}^{f_{\mathbf{a}}}=c(\mathbf{h})$. It follows from Lemma \ref{lemma1} that $\qlie_{\Phi}^{f_{\mathbf{a}}}$ is a commutative Lie subalgebra of $\qlie_{\Phi}$ consisting of semi-simple elements. Therefore, there exists a vector subspace $\rlie$ of $\qlie_{\Phi}$ such that $\qlie_{\Phi}=\qlie_{\Phi}^{f_{\mathbf{a}}}\oplus \rlie$ and $[\qlie_{\Phi}^{f_{\mathbf{a}}},\rlie]\subset\rlie$. So 
$[\qlie_{\Phi},\qlie_{\Phi}^{f_{\mathbf{a}}}]\subset\rlie$ and the result follows by 
Theorem \ref{stable}.
\end{proof}

We shall now show that $\qlie_{\Phi}$ does not necessarily contain a stable linear form in general.

Let ${\mathcal{M}}_{6,6}(\kset )$ be the set of $6\times6$ matrices and let 
$\{E_{i,j},1\leqslant i,j\leqslant 6\}$ be its canonical basis. Let $\glie$ be the subset of 
${\mathcal{M}}_{6,6}(\kset )$ whose trace is equal to zero. 
We fix $\hlie$ the set of diagonal matrices of $\glie$ and $\blie$ the set of upper triangular 
matrices of $\glie$.

Therefore, we can choose $X_{\alpha_i+\dots+\alpha_j}=E_{i,j+1}$, for 
$1\leqslant i\leqslant j\leqslant 5$. Set 
$$
\Phi=\{\alpha\in\Delta^+;\alpha\geqslant \alpha_1+\alpha_2+\alpha_3 \mbox{ or }\alpha\geqslant\alpha_3+\alpha_4+\alpha_5\}.
$$
The Panov H-sequence of $\Delta^+\setminus\Phi$ is 
$$
\mathbf{h}=(\alpha_1+\alpha_2,\alpha_2+\alpha_3+\alpha_4,\alpha_3+\alpha_4,\alpha_3,\alpha_4+\alpha_5,\alpha_5).
$$ 
We have $\ell(\mathbf{h})=6$, $d(\mathbf{h})=5$ and $c(\mathbf{h})=1$.

By definition, $\ilie=\glie^{\Phi}$ is an ad-nilpotent ideal of $\glie$. 

\begin{prop}\label{nostable}
The Lie algebra $\qlie_{\Phi}^*$ does not possess any stable linear form.
\end{prop}

\begin{proof}
By Theorem \ref{th_indice} or Proposition \ref{prop01}, 
we have $\chi(\qlie_{\Phi})=1$. Let $t\in\kset^*$ and 
$\boldsymbol{\lambda}_t=(1,\dots,1,t)\in\kset^6$. Set 
$f_{\boldsymbol{\lambda}_t}=\sum_{i=1}^5 X_{\theta_i}+tX_{\theta_6}$ and
$$
Z_t=X_{\alpha_1}-X_{\alpha_3}+\frac{1}{t}X_{\alpha_5}+(1+\frac{1}{t})X_{\alpha_2+\alpha_3}+X_{\alpha_3+\alpha_4}-X_{\alpha_4+\alpha_5}.
$$
A simple calculation gives $\qlie_{\Phi}^{f_{\boldsymbol{\lambda}_t}}=\mathrm{Vect}(Z_t).$

Let $H\in\hlie$ be such that $\alpha_1(H)=\alpha_3(H)=\alpha_5(H)=1$ and $\alpha_2(H)=\alpha_4(H)=0$. Then $[H,Z_t]=Z_t$, so $[\qlie_{\Phi},\qlie_{\Phi}^{f_{\boldsymbol{\lambda}_t}}]\cap
\qlie_{\Phi}^{f_{\boldsymbol{\lambda}_t}}\not=\{0\}$. 
By Theorem \ref{stable}, $f_{\boldsymbol{\lambda}_t}$ is not stable. 

\medskip
Denote by $Q$ the algebraic adjoint group of $\qlie_{\Phi}$. Then $Q$ can be identified
with the quotient of the set of invertible upper triangular matrices by a closed normal subgroup. 
Let $s,t\in\kset^*$. Assume that $\qlie_{\Phi}^{f_{\boldsymbol{\lambda}_t}}$ and $\qlie_{\Phi}^{f_{\boldsymbol{\lambda}_s}}$ are $Q$-conjugate, then there exist $\lambda\in\kset^*$ and 
an invertible triangular matrix $P$ such that $PZ_sP^{-1}-\lambda Z_t\in \ilie$. By the definition of $\ilie$, for any element $L\in\ilie$ and for any upper triangular matrix $R$, we have $LR\in\ilie$. It follows that $PZ_s-\lambda Z_tP\in\ilie$. By a direct computation, we obtain that $t=s$.

Recall that for $f,g\in\qlie_{\Phi}^*$, if $f$ and $g$ are $Q$-conjugate, then $\qlie_{\Phi}^f$ and 
$\qlie_{\Phi}^g$ are also $Q$-conjugate. We define 
$$
\begin{array}{ccclcr}
\Psi &: &Q\times \kset^* &\rightarrow& \qlie_{\Phi}^* \\
&&(x,t)&\mapsto &x.f_{\boldsymbol{\lambda}_t} \\
\end{array}
$$
By the above consideration, we deduce that 
$$
\Psi^{-1}(f_{\boldsymbol{\lambda}_t})=\{(x,s)\in Q\times \kset^*; x.f_{\boldsymbol{\lambda}_s}=f_{\boldsymbol{\lambda}_t}\}=\mathrm{Stab}_Q(f_{\boldsymbol{\lambda}_t}).
$$ 

Since $\dim \mathrm{Im}\Psi=\dim(Q\times\kset^*)-\dim \Psi^{-1}(f_{\boldsymbol{\lambda}_t})$, 
we have $\dim \mathrm{Im}\Psi=\dim Q$. Therefore $Q.\kset^*$ contains a non-empty open subset 
of $\qlie_{\Phi}^*$ which does not contain any stable linear form.

Since the set of stable linear forms of $\qlie_{\Phi}$ is an open subset of $\qlie_{\Phi}^*$, the result follows immediately.
\end{proof}

%----------------------
\section{Remarks on the exactness of the upper bounds}

Assume that $\glie$ is of type $C_7$. Using the numbering of simple roots in \cite{TYL}, set
$$
\begin{array}{ll}
\beta_1=\alpha_1+\alpha_2+\alpha_3+\alpha_4+\alpha_5, \\
\beta_2=\alpha_2+\alpha_3+\alpha_4+2\alpha_5+2\alpha_6+\alpha_7, \\
\beta_3=2\alpha_4+2\alpha_5+2\alpha_6+\alpha_7,\\
\Phi=\{\alpha\in\Delta^+; \alpha \geqslant \beta_i,\mbox{ for some } i \mbox{ such that } 1\leqslant i\leqslant 3\}.
\end{array}
$$
We check by hand that the minimal length of an H-sequence associated to $\Delta^+\setminus\Phi$ 
is $8$. 
For example, the following H-sequence  $\mathbf{h}=(\theta_1,\dots ,\theta_s)$  is of length $8$ : 
$$
\begin{array}{ll}
\theta_1=\alpha_1+\alpha_2+\alpha_3+\alpha_4, &\theta_5=\alpha_4+2\alpha_5+2\alpha_6+\alpha_7,\\
\theta_2=\alpha_2+\alpha_3+\alpha_4+\alpha_5+2\alpha_6+\alpha_7,& \theta_6=2\alpha_5+2\alpha_6+\alpha_7,\\
\theta_3=\alpha_3+\alpha_4+\alpha_5+\alpha_6+\alpha_7,& \theta_7=\alpha_5+\alpha_6,\\
\theta_4=\alpha_2+\alpha_3+\alpha_4+\alpha_5, &\theta_8=\alpha_5.\\
\end{array}
$$
and $d(\mathbf{h})=7$.

By Theorem \ref{th_index}, we have 
$$
\chi(\qlie_{\Phi})\leqslant\ell+\ell(\mathbf{h})-2d(\mathbf{h})=1, \mbox{ and  } \chi(\mlie_{\Phi})\leqslant \ell(\mathbf{h})=8.
$$
So $\chi(\qlie_{\Phi})=1$ by Proposition \ref{prop01}. But by considering an arbitrary linear form, we found that $\chi(\mlie_{\Phi})\leqslant 6$. Thus the upper bound for the index of $\mlie_{\Phi}$ is not always exact when $\glie$ is not of type $A$.

\medskip
We did some computations using \textsc{Gap4} on arbitrary linear forms when $\glie$ is of rank less than or equal to $6$, and we have not found an example where the upper bound for $\chi(\qlie_{\Phi})$ is not exact. This leads us to formulate the following conjecture :

\begin{conjecture}
Let $\Phi\subset\Delta^+$ such that $\glie^{\Phi}$ is an ad-nilpotent ideal of $\glie$. There exists an H-sequence  $\mathbf{h}$ of $\Delta^+\setminus\Phi$ such that
$$
\chi(\qlie_{\Phi})=\ell+\ell(\mathbf{h})-2d(\mathbf{h}).
$$
\end{conjecture}


\begin{thebibliography}{10}

\bibitem{Char}  \textsc{J.-Y. Charbonnel},  Propri\'et\'es (Q) et (C). Vari\'et\'e commutante,
\emph{Bull. Soc. Math. France}, 132 (2004) 477-508.

\bibitem{DSV} \textsc{P. Damianou, H. Sabourin and P. Vanhaecke}, work in progress.

\bibitem{DK}  \textsc{V. Dergachev and A. Kirillov},
Index of Lie algebras of seaweed type, \emph{J. of Lie Theory}, 10 (2000) 331-343.

\bibitem{Jos} \textsc{A. Joseph},
On semi-invariants and index for biparabolic (seaweed) algebras I, \emph{J. of Algebra},  
305  (2006) 487-515. 

\bibitem{Mor1}  \textsc{A. Moreau}, Indice du normalisateur du centralisateur d'un 
\'el\'ement nilpotent dans une alg\`ebre de Lie semi-simple, 
\emph{Bull. Soc. Math. France}, 134  (2006) 83-117.

\bibitem{Mor2}  \textsc{A. Moreau}, Indice et d\'ecomposition de Cartan d'une alg\`ebre de Lie 
semisimple r\'eelle, \emph{J. of Algebra},  303  (2006)  382-406.

\bibitem{Pa} \textsc{A.N. Panov}. \emph{On index of certain nilpotent Lie algebras}. arXiv:0801.3025v1.

\bibitem{Panyseaweed}  \textsc{D. Panyushev}, Inductive formulas for the index of seaweed
Lie algebras, \emph{Moscow Math. Journal}, 1 (2001) 221-241.

\bibitem{Pany}  \textsc{D. Panyushev}, The index of a Lie algebra, the centralizer of a nilpotent element, 
and the normalizer of the centralizer,
\emph{Math. Proc. Camb. Phil. Soc.}, 134 (2003) 41-59.

\bibitem{TY} \textsc{P. Tauvel, R.W.T.Yu}. \emph{Indice et formes lin\'eaires stables dans les 
alg\`ebres de Lie}. J. Algebra 273 (2004) 507-516.

\bibitem{TY2} \textsc{P. Tauvel, R.W.T.Yu}. \emph{Sur l'indice de certaines alg\`ebres de Lie}. Ann. Inst. Fourier (Grenoble) 54 (2004), no. 6, 1793-1810.

\bibitem{TYL} \textsc{P. Tauvel, R.W.T.Yu}. \emph{Lie algebras and algebraic groups}. Springer Monographs in Mathematics. Springer-Verlag, 2005.

\bibitem{Yaki}  \textsc{O. Yakimova}, The index of centralisers of elements in classical Lie algebras,
\emph{Funct. Analysis and its Applications}, 40 (2006) 42-51.
\end{thebibliography}
\end{document}